\documentclass {amsart}
\usepackage[french,english]{babel}
\usepackage{graphicx}
\usepackage {appendix}
\usepackage{amsfonts}
\usepackage{url}
\usepackage{hyperref} 
\hypersetup{backref,pdfpagemode=FullScreen,colorlinks=true}
\usepackage{amsmath}
\usepackage{amssymb}
\usepackage{amscd}
\usepackage{color}
\usepackage{amsthm}
\usepackage{bm}
\usepackage{indentfirst}
\usepackage[hmargin=3cm,vmargin=3cm]{geometry}
\usepackage[all, cmtip]{xy}
\numberwithin{equation}{section}

\newtheorem{theorem}[equation]{Theorem}

\newtheorem{lemma}[equation]{Lemma}

\theoremstyle{definition}

\theoremstyle{remark}

\newcommand\vol{\operatorname{vol}}

\DeclareMathOperator {\area} {area}
\DeclareMathOperator {\diam} {diam}

\DeclareMathOperator{\image}{\mathrm{image}}

\DeclareMathOperator{\lcs}{l.c.s.}

\begin{document}
\title{Mean curvature versus diameter and energy quantization}
\author{Yasha Savelyev}
\thanks {Partially supported by PRODEP grant of SEP, Mexico}
\email{yasha.savelyev@gmail.com}
\address{University of Colima, CUICBAS}
\begin{abstract} We first partially extend a theorem of Topping, on the relation between mean curvature and intrinsic diameter, from immersed submanifolds of $\mathbb{R} ^{n} $ to almost everywhere immersed, closed submanifolds of a compact Riemannian manifold. We use this to prove quantization of energy for pseudo-holomorphic closed curves, of all genus, in a compact locally conformally symplectic manifold.   
%
\end{abstract}

 \maketitle
\section {Introduction}

 For a compact  almost complex manifold $(M, J)$ with metric $g$, energy quantization is a statement that there is a $\hbar (M, g, J)>0$ so that any non-constant pseudo-holomorphic map of a sphere into $M$ has energy at least $\hbar (M,g, J)$.

In order to carry out Gromov-Witten theory for locally conformally symplectic manifolds, or $\lcs$ manifolds for short, we need in general a stronger form of energy quantization of pseudo-holomorphic curves, because it may be necessary to work with trivial homology classes (cf. \cite{citeSavelyevConformalSymplectic}).  In this case, if we are working with a higher genus curve, the energy of the  curve a priori can collapse to 0, which would break compactness and invariance arguments. This is the ``black earth'' analogue of the ``blue sky catastrophe'' cf. \cite{citeShilnikovTuraevBlueSky}, \cite{citeFullerBlueSky}, \cite{citeSavelyevFuller}. We show that this cannot happen.

The classical energy quantization is proved using mean value inequalities for pseudo-holomorphic maps, \cite{citeMcDuffSalamon$J$--holomorphiccurvesandsymplectictopology}, \cite{citeKatrinQuantization}.
To prove quantization of energy for higher genus curves in a $\lcs$ manifold, we need more sophisticated differential geometry. In this note we first partially extend a theorem of Topping relating diameter and mean curvature of immersed submanifolds of $\mathbb{R} ^{n} $, and then use this to prove our quantization result.
\subsection {Mean curvature vs diameter}  In \cite{citeToppingRelatingDiameter} Topping gave via a concise but sophisticated argument, partially based on ideas of Ricci flows, a simple relation between intrinsic diameter and mean curvature, for immersed submanifold of $\mathbb{R} ^{n} $.  Let us state it here: \begin{theorem} [\cite{citeToppingRelatingDiameter}] 
  For $\Sigma ^{m} $ a smoothly immersed closed submanifold of $\mathbb{R} ^{n} $ we have:
\begin {equation*}    \diam (\Sigma)  \leq Const (m)
    \int _{\Sigma} |{\textbf{H}}| ^{m-1}  \, dvol,
 \end{equation*}  
for $\textbf{H} $ the mean curvature vector field along $\Sigma$, $vol$ the volume measure induced by the standard ambient metric, and $\diam$ the intrinsic diameter:  $$\max _{x,y \in \Sigma} dist_{(\Sigma, g _{st}) } 
   (x, y),$$  
for $g _{st}$ the metric on $\Sigma$ induced by the standard metric on $\mathbb{R} ^{n} $.
\end{theorem}                                                                   We generalize this to almost everywhere immersed submanifolds of general compact Riemannian manifolds, using Nash embedding theorem. It is indeed not obvious what the right generalization should be. Ultimately our hand is forced by our intended application, but our generalization does seem natural from at least one point of view.

In plain words the statement is the following: if the volume of a given closed immersed submanifold of a compact Riemannian manifold is ``small'' but diameter ``large'' then the mean curvature must be somewhere large. However we suitably extend this to non-immersed submanifolds, so the actual formulation is more involved.

Let $(M,g)$ be a smooth Riemannian manifold.
Denote by ${S} (M, \Sigma)$ the space of smooth maps $u$ of a closed Riemannian $m$-fold $\Sigma$ into $(M,J)$. 
Define:                                                                          \begin{equation*}
Vol (u) = \int _{\Sigma} |du| _{g}  \, dvol _{\Sigma},
\end{equation*}
to be the volume functional, for $vol _{\Sigma} $ the volume measure on $\Sigma$ for a fixed auxiliary metric $g _{aux} $ on $\Sigma$, and $|du| _{g} $ the operator norm of $du$ with respect to $g _{aux} $ and $g$.

For $u \in S (M, \Sigma)$ we say that it is \textbf{\emph{non-singular}} at $p \in \Sigma$ if $du (p)$ is non-singular.
Let $$S (C) = S (C, M, \Sigma) \subset {S} (M, \Sigma)$$ consist of elements $u$,  non-singular on a subset of full $vol _{\Sigma} $ measure and satisfying:   \begin{equation*}
|\textbf{H} (u (p))| < C, 
\end{equation*}                                                                 for all $p$ s.t. $u$ is non-singular at $p$, where $\textbf{H} (u (p))$ is the mean curvature vector at $u (p)$. In this case $Vol$ is continuously differentiable at $u$.
\begin{theorem} \label{thm:main}  
   Let $\diam(u)$ denote the (intrinsic) diameter in $(M,g)$ of $\image (u)$.
   Then for all $u \in S (C)$ $$\diam(u) \leq F (g,C,m) Vol(u), $$ for some function $F$.
\end{theorem}

\subsection {Application to energy quantization} 
A locally conformally symplectic manifold or $\lcs$ manifold, is a smooth $2n$-fold $M$, with a non-degenerate $2$-form $\omega$, which is locally diffeomorphic to $e ^{f} \omega _{0}  $ for some functions $f$, and $\omega _{0} $ the standard symplectic form on $\mathbb{R} ^{2n} $.
There has been some interest recently in developing Gromov-Witten theory for $\lcs$ manifolds, \cite{citeSavelyevConformalSymplectic}, \cite{citeMurphyConformalsymp}.
This is partly impeded by our lack of understanding of how badly holomorphic curves can behave in such a manifold.  For instance, in the context of holomorphic curves in symplectic manifolds, we have the following energy quantization phenomenon: 
\begin{theorem} \label{thm:quantization:symplectic}  Given a closed symplectic $(M,\omega)$, $J$ an $\omega$-tamed almost complex structure, meaning $\omega (\cdot, J \cdot) >0$, there exists a constant $\hbar=\hbar (\omega,J)$, s.t. for any $u: (\Sigma,j) \to M$ a non-constant $J$-holomorphic map of a closed Riemann surface $(\Sigma,j)$, the energy of the map satisfies:
\begin{equation*}
e (u) = \int _{\Sigma} |du| ^{2} _{g _{J} }  \, dvol _{\Sigma}  = \int _{\Sigma} u ^{*} \omega \geq \hbar,
\end{equation*}
for $g _{J} $ the metric $\omega (\cdot, J \cdot) $.
\end{theorem}
For $\Sigma=S ^{2} $ this holds via a generalized mean value inequality, \cite{citeMcDuffSalamon$J$--holomorphiccurvesandsymplectictopology}, \cite{citeKatrinQuantization} and the symplectic condition on $M$ can be loosened to just almost complex. For more general $\Sigma$, but with $M$ symplectic this is not previously known, but we give here a very simple,  given state of the art, argument via geometric measure theory. 

We then directly generalize this to $\lcs$ manifolds. 
\begin{theorem}  \label{thm:quantization} Given a closed $\lcs$ $(M,\omega)$, $J$ an $\omega$-compatible almost complex structure there exists a constant $\hbar=\hbar (\omega,J)$, s.t. for any $u: (\Sigma,j) \to M$ a non-constant $J$-holomorphic map of a closed Riemann surface $(\Sigma,j)$, the energy of the map satisfies:
\begin{equation*}
e (u) = \int _{\Sigma} |du| ^{2} dvol _{\Sigma}  = \int _{\Sigma} u ^{*} \omega \geq \hbar,
\end{equation*}
for $|du|$ the operator norm with respect $g _{aux} $ and $g _{J}= \omega (\cdot, J \cdot) $.
\end{theorem} 
\section {Proofs}
\begin{proof} [Proof of Theorem \ref{thm:quantization:symplectic}] Suppose otherwise, then we have a sequence $\{u _{i}   \}$ of $J$-holomorphic curves with $e (u _{i} ) \to 0$ as $i \to \infty$.
In particular, $|u _{i}| \to 0$, where $|\cdot|$ is the mass norm, and $u _{i} $ are understood as integral 2-currents, Federer~\cite{citeFedererGeometricMeasure}.
By the main compactness theorem for currents, $\{u _{i} \}$ has a convergent subsequence $\{u _{i _{k} } \}$ to an integral 2-current with mass necessarily $0$, and hence $0$ in the vector space of closed integral 2-currents $\mathcal{I} _{2} (M) $.  Next it is proved in \cite{citeFedererGeometricMeasure}, that the space $\mathcal{H} _{k} (M) $ of closed integral $k$-currents modulo exact integral $k$-currents is discreet, with respect to the topology induced by the mass norm, and is isomorphic to singular integral $k$-homology. Moreover the natural map $$q _{k} : \mathcal{I} _{k} (M)  \to \mathcal{H} _{k} (M) $$ is continuous. Thus $\{q _{2}  (u _{i _{k} } )\}$ is eventually constant, which means that $u _{i _{k} } $ are eventually in class $0$, which is a contradiction, since all $u _{i} $ have positive symplectic area.
\end{proof}
\begin {proof}[Proof of Theorem \ref{thm:main}] For $u \in S (C,M, \Sigma)$
let $dVol (u)$ denote the differential of $Vol$ at $u$, and $|dVol (u)|$ its operator norm, for the induced by $g$ ``Riemannian metric'' on $S (M, \Sigma)$, which is understood as a Frechet manifold.  Explicitly the ``Riemannian metric'' is:
\begin{equation*}
 \langle V _{1}, V _{2}   \rangle = \int _{\Sigma}  \langle V _{1}  (p), V _{2}  (p) \rangle _{g} \, dvol _{\Sigma},     
\end{equation*}
for $V _{i}  \in T _{u} S (M, \Sigma) $, i.e. $V _{i}   $ are smooth maps $$V_i: \Sigma \to TM$$ satisfying $$V _{i}  (p) \in T _{u(p)} M,
$$ we call such a $V _{i} $ a \textbf{\emph{vector field along $u$}}.

 Given $u \in S (C, M, \Sigma)$, non-singular on $U \subset \Sigma$, and $V=fN$ for $N$ the normal unit vector field along $U$, and $f$ a smooth function with support in $U$,
we have:                                                                         \begin{equation*} |dVol (u) (V)| = \int _{\Sigma} mf \textbf{H} \,dvol _{\Sigma},    \end{equation*} 
by classical differential geometry.

It follows since $u$ is non-singular on a set of full measure, and $\textbf{H}$ being bounded, that $dVol$ is continuously differentiable at $u$ and                                              \begin{equation} \label{eq:energyMeancurvature} |dVol (u)| \leq m|\textbf{H}|,  
\end{equation} 
where the right hand side is the supremum over $p \in \Sigma$ s.t. $u$ is non-singular at $p$, of the magnitude of the mean curvature $|\textbf{H} (u(p))|$.

 Pick an isometric Nash embedding $N$ of $ (M,g)$ into
 $\mathbb{R} ^{n} $, where $n$ is large enough.
\begin{lemma} 
For all $u \in S (C, M, \Sigma)$:
 \begin{equation*}
 |dVol (N \circ u)| < C',
\end{equation*}
for some $C'$, independent of $u$.
\end{lemma} 
   \begin{proof} Let $U \subset \Sigma$ be an open set on which $u$ is non-singular.
      In what follows we conflate the notation for $U$  and its images $u (U)$, $N \circ u (U) $. In other words we just think in terms of subspaces $U \subset M \subset \mathbb{R} ^{n} $.
  Let  $h$ be the second fundamental form on $T _{p} U $:
     \begin{equation*}
  h (v,w) = \widetilde{\widetilde{ \nabla}} _{v} w - \nabla _{v} w,
     \end{equation*}
      where $\widetilde{\widetilde{\nabla}}$ is the Levi-Civita connection of $(\mathbb{R} ^{n}, g _{st})  $, $\nabla  $ is the Levi-Civita connection of the submanifold $U \subset \mathbb{R} ^{n} $ and where we locally extend $v,w \in T _{p} U $ to vector fields tangent to $U$.
      The mean curvature vector at $p \in U \subset \mathbb{R} ^{n}  $  is given by: 
  \begin{equation*}
     \textbf{H} _{\mathbb{R} ^{n} }  (p)=\frac{1}{m} \sum _{i} h (e _{i}, e _{i}),
     \end{equation*}
     where $\{e _{i} \}  $ is an orthonormal basis for $T _{p} U $.
     Likewise $\widetilde{ \nabla}$ will denote in what follows the Levi-Civita connection of $(M,g)$.  So we have:   
\begin{align*}
   m|\textbf{H} _{\mathbb{R} ^{n} }  (p)|= \sum _{i} h (e _{i}, e _{i})(p) |= |\sum _{i}  (\widetilde{\widetilde{ \nabla}} _{e _{i} } e _{i}  - \nabla _{e _{i} } e _{i}) | &  = |\sum _{i} (\widetilde{\widetilde{ \nabla}} _{e _{i} } e _{i}  - \widetilde{\nabla} _{e _{i} } e _{i} + \widetilde{\nabla} _{e _{i} } e _{i} - \nabla _{e _{i} }  e _{i}) | \\ & \leq    |\sum _{i}  ({\widetilde{ \nabla}} _{e _{i} } e _{i}  - \nabla _{e _{i} } e _{i}) | + mB \\
& \leq mC+mB,
\end{align*}
where $$B= \sup _{e \in TM, |e|=1} |\widetilde{\widetilde{ \nabla}} _{e  } e   - \widetilde{ \nabla} _{e  } e |.
$$
The lemma then follows by \eqref{eq:energyMeancurvature}. 
   \end{proof} 


Now let $u'$ be $C ^{\infty}
   $ $\delta$-close to $N \circ u$  s.t. $u'$ is immersed, for some $\delta>0$ sufficiently small so that $|dVol (u')| < C'$ for $C'$ as in the lemma above. We can find such a $u'$ since $Vol$ is continuously differentiable at $N \circ u$.
By \eqref{eq:energyMeancurvature}
\begin{equation*}
|\textbf{H} _{\mathbb{R} ^{n} }  (p)| < \frac{C'}{m},
\end{equation*} 
for $\textbf{H} _{\mathbb{R} ^{n}  } $ the mean curvature vector field along $u'$, in $\mathbb{R} ^{n} $.

Since $\Sigma$ is closed we get by Topping's theorem: 
 \begin{equation*}
    \diam (u' (\Sigma))  \leq Const (m)
    \int _{\Sigma} |{\textbf{H}} _{\mathbb{R} ^{n} }   | ^{m-1}  \, dvol _{\Sigma}.
 \end{equation*}  
   By the lemma above the function $| \textbf{H} _{\mathbb{R} ^{n} }  |$ on $\Sigma$ is universally (independently of $u$) bounded from above by some $C'$.
So we get:
 \begin{align*}
     \diam (u' (\Sigma))  \leq  Const (m) \cdot (\frac{C'}{m}) ^{m-1}  \cdot \vol (\Sigma),  
 \end{align*}  
passing to the limit as $u' \to N \circ u$ we get  the required inequality.  \end{proof}

\begin{proof} [Proof of Theorem \ref{thm:quantization}] First we need the following.  
\begin{theorem} \label{lemma:boundedmeancurvature}
For $u: \Sigma \to M$ a non-constant $J$-holomorphic map of a Riemann surface into an $\lcs$ $(M, \omega)$, as above, with $J$ $\omega$-compatible, and with $\area (u) <D$:              \begin{equation*}
|dVol (u)| < C (\omega, J, D),
\end{equation*}                                                                              for a constant $C (\omega, J, D)>0$, independent of $u$.
\end{theorem} 
\begin{proof} The map $u$, being $J$-holomorphic, is non-singular outside of a set of finitely many points, McDuff-Salamon~\cite{citeMcDuffSalamon$J$--holomorphiccurvesandsymplectictopology}.
We may fix a finite cover of $M$ by charts $\phi _{i}: U _{i} \subset \mathbb{R} ^{2n} \to M   $, $\phi _{i}^{*} \omega = e^{f _{i}} \omega _{0} 
$ with $\omega _{0} $ the standard symplectic form, and with $U _{i} $ contractible. 
 Then 
$\phi _{i} ^{-1} \circ u  $ is a $\phi _{i} ^{*} J
$-holomorphic map defined on $u ^{-1} (\phi _{i} (U _{i}) ) $ into $\mathbb{R}^{2n} $, and $\phi _{i} ^{*} J$ is
compatible with $f _{i} \omega _{0} $ and hence with $\omega _{0} $.
Fix a cover of $\Sigma $ by disk domains $\{U _{i_j} \}$, with each
$U _{i _{j} } \subset  u ^{-1} (\phi _{i} (U _{i}) ) $ for some $i$. Then ${u _{i,j}}= \phi
_{i} ^{-1} \circ   u |
_{U _{i _{j} } } $ is a non-constant  $\phi _{i} ^{*} J$-holomorphic
curve in  $\mathbb{R} ^{2n} $. 
   
   Since $g=(\omega _{0}, \phi _{i} ^{*} J )$ is an almost Kahler metric,  $$dVol( {u} _{i,j}) (V) =0 $$ for every vector field $V$, along $u _{i,j} $, vanishing on a neighborhood of the boundary of $U _{i _{j} } $, and near the singularities,
   since $\omega _{0}  $ is then a calibration, Harvey-Lawson~\cite{citeLawsonHarveyCalibrated}.  As  the cover $\{U _{i} \}$ is finite,  the $C ^{\infty} $ norm of the functions $f
_{i} $
is universally bounded: $|f _{i} | <B$, for some some $B$ and all $i$.
It readily follows that                                                                 $$|dVol( {u} _{i,j}) (V)| <  C (\omega, J,D), $$
for $V$ as above with $|V|=1$ and with $Vol$ computed with respect to the metric
    $$g'=(e^{f _{i}}  \omega _{0} , \phi ^{*}_{i} J),$$ and some $C (\omega, J,D)$.
This is because the distortion of derivative $dVol$ corresponding to the bounded conformal distortion $e ^{f _{i} }$ can be easily bounded in terms of $B, D$. 
   More explicitly, let $\{u _{t} \} $  be a one parameter family of maps $U _{i_j} \to \mathbb{R} ^{2n} $, with $\frac{d}{dt} \vert _{t=0} u _{t} = V  $. 
Then $$ |dVol({u} _{i,j}) (V)| = |\frac{d}{dt} \vert _{t=0} \int _{V _{i,j} } | du _{t}| _{g'}  \,dvol _{\sigma}| =   |\frac{d}{dt} \vert _{t=0} \int _{V _{i,j} }  e ^{f _{i} } (u _{t} ) |du _{t}| _{g}  \,dvol _{\sigma}|,   $$
Now pull the derivative inside the integral and use product rule.
   The result readily 
follows from this.
\end{proof}
Now let $\epsilon$ be the Lebesgue covering number of $\{U _{i} \}$ with
respect to the metric $(\omega,J)$.
Combining Theorem \ref{lemma:boundedmeancurvature}, Theorem
\ref{thm:main} and \eqref{eq:energyMeancurvature} we get that for $u$ non-constant as in the hypothesis if $\area
(u) < \hbar$ then $\diam (u) < \epsilon$, for some $\hbar$ independent
of $u$. Consequently the image of $u$ is contained in some $U _{i} $,
and so $\phi _{i} ^{-1} \circ u$ is a $\phi _{i} ^{*} J
$-holomorphic map of a sphere into the almost Kahler contractible manifold $(U
_{i} , \omega _{0},  \phi _{i} ^{*} J )$ and so must be constant, which is a contradiction.
\end{proof}
\
\subsection* {Acknowledgements} I am grateful to Peter Topping, and Egor Shelukhin for comments on earlier ideas.
\bibliographystyle{siam}  
\bibliography{/root/texmf/bibtex/bib/link} 

\begin{thebibliography}{10}

\bibitem{citeMurphyConformalsymp}
{\sc C.~{Baptiste} and A.~{Murphy}}, {\em {Conformal symplectic geometry of
  cotangent bundles}}, {arXiv},  (2016).

\bibitem{citeFedererGeometricMeasure}
{\sc H.~{Federer}}, {\em {Geometric measure theory. Repr. of the 1969 ed.}},
  Berlin: Springer-Verlag, repr. of the 1969 ed.~ed., 1996.

\bibitem{citeFullerBlueSky}
{\sc F.~{Fuller}}, {\em {Note on trajectories in a solid torus.}}, {Ann. Math.
  (2)}, 56 (1952), pp.~438--439.

\bibitem{citeLawsonHarveyCalibrated}
{\sc R.~{Harvey} and H.~B. {Lawson}}, {\em {Calibrated geometries.}}, {Acta
  Math.}, 148 (1982), pp.~47--157.

\bibitem{citeMcDuffSalamon$J$--holomorphiccurvesandsymplectictopology}
{\sc D.~McDuff and D.~Salamon}, {\em $J$--holomorphic curves and symplectic
  topology}, no.~52 in American Math. Society Colloquium Publ., Amer. Math.
  Soc., 2004.

\bibitem{citeSavelyevFuller}
{\sc Y.~Savelyev}, {\em {Extended Fuller index, sky catastrophes and the
  Seifert conjecture}}, International Journal of mathematics, to appear.

\bibitem{citeSavelyevConformalSymplectic}
\leavevmode\vrule height 2pt depth -1.6pt width 23pt, {\em {Gromov Witten
  theory of a locally conformally symplectic manifold and the Fuller index}},
  arXiv,  (2016).

\bibitem{citeShilnikovTuraevBlueSky}
{\sc A.~{Shilnikov}, L.~{Shilnikov}, and D.~{Turaev}}, {\em {Blue-sky
  catastrophe in singularly perturbed systems.}}, {Mosc. Math. J.}, 5 (2005),
  pp.~269--282.

\bibitem{citeToppingRelatingDiameter}
{\sc P.~{Topping}}, {\em {Relating diameter and mean curvature for submanifolds
  of Euclidean space.}}, {Comment. Math. Helv.}, 83 (2008), pp.~539--546.

\bibitem{citeKatrinQuantization}
{\sc K.~{Wehrheim}}, {\em {Energy quantization and mean value inequalities for
  nonlinear boundary value problems.}}, {J. Eur. Math. Soc. (JEMS)}, 7 (2005),
  pp.~305--318.

\end{thebibliography}
\end{document}